\documentclass[leqno,12pt,a4paper]{amsart}%
\usepackage{amsmath,amssymb,amsfonts}
\usepackage{color}
\usepackage{mathabx}
\usepackage{hyperref}

\newtheorem{prop}{Proposition}[section]
\newtheorem{teo}{Theorem}[section]
\newtheorem{lema}{Lemma}[section]
\newtheorem{coro}{Corollary}[section]

\newtheorem{rem}{Remark}[section]

\def\ep{\varepsilon}
\def\ul{u_\lambda}
\def\la{\lambda}

\def\ep{\varepsilon}
\def\R{{\mathbb{R}}}

\def\T{{\mathcal T}}
\def\a{\mathfrak{a}}
\def\L{{\mathcal L}}

\begin{document}
	
\title[Asymptotics in a nonlocal heat equation with absorption]{\bf Large-time behavior in a nonlocal heat equation with absorption. The absorption dominated case with fast decaying initial data}
	
\author[C. Cort\'{a}zar,  F. Quir\'{o}s \and N. Wolanski]{Carmen Cort\'{a}zar,  Fernando  Quir\'{o}s, \and Noem\'{\i} Wolanski}

\address{Carmen Cort\'{a}zar\hfill\break
\indent Academia Chilena de Ciencias\hfill\break
\indent Almirante Montt 454, Santiago de Chile, Chile.}
\email{{\tt ccortazar2357@gmail.com} }

\address{Fernando Quir\'{o}s\hfill\break
\indent Departamento  de Matem\'{a}ticas, Universidad Aut\'{o}noma de Madrid,\hfill\break\indent and Instituto de Ciencias Matem\'aticas ICMAT (CSIC-UAM-UCM-UC3M),\hfill\break
\indent 28049-Madrid, Spain.}
\email{{\tt fernando.quiros@uam.es} }

\address{Noem\'{\i} Wolanski \hfill\break
\indent IMAS-UBA-CONICET, \hfill\break\indent Ciudad Universitaria, Pab. I,\hfill\break
\indent (1428) Buenos Aires, Argentina.}
\email{{\tt wolanski@dm.uba.ar} }
	
	

\date{}

\begin{abstract}
We study the large-time behavior of nonnegative solutions to a nonlocal dispersal equation in $\R^N$ with an absorption term modeled by $-u^p$, with $1<p<1+\frac2N$. The initial datum $u_0$ is assumed to be bounded, and to satisfy $|x|^{\frac2{p-1}}u_0(x)\to A\ge0$ as $|x|\to\infty$. Under these assumptions, we prove that the decay rate is that of the purely absorbing problem, while the limit profile is a very singular solution to a local diffusion problem with absorption if $A=0$, and a solution to this same local problem with initial datum $A|x|^{-\frac2{p-1}}$ if $A>0$.
\end{abstract}
	
\maketitle
	
\section{Introduction}
\setcounter{equation}{0}

We study the large-time behavior of nonnegative solutions to
\begin{equation}\label{eq-problem without sign}
        \begin{array}{c}
        \displaystyle u_t-\L u=-|u|^{p-1}u\quad\mbox{in }\R^N\times(0,\infty),\qquad u(x,0)=u_0(x)\quad\mbox{in }\R^N,\\
        \displaystyle\text{where }\L u(x)=\int J(x-u)\big(u(y)-u(x)\big)\,{\rm d}y,
        \end{array}
\end{equation}
with $1<p<1+\frac2N$, and $0\le u_0\in  L^\infty(\R^N)$ such that
\begin{equation}\label{eq-condition at infinity}
        |x|^{\frac2{p-1}}u_0(x)\to A\ge0\quad\mbox{as }|x|\to\infty.
\end{equation}
In particular, $u_0\in L^1(\mathbb{R}^N)\cap L^\infty(\mathbb{R}^N)$. The assumptions on the kernel $J$ are:  $J\in C^\infty(\R^N)$ radially symmetric, $J>0$ in $B_d$, $J=0$ in $\mathbb{R}^N\setminus B_d$ for some $d>0$, and $\int J(x)\,{\rm d}x=1$. By a solution of the problem we mean a function
$$
    u\in C([0,\infty);L^1(\R^N))\cap  C^1(0,\infty;L^1(\R^N))\cap  L^\infty(0,\infty;L^\infty(\R^N))
$$
that satisfies~\eqref{eq-problem without sign} in a pointwise sense. We define  a sub-solution (resp. super-solution) of the problem  in a similar way, as a function in these spaces that satisfies \eqref{eq-problem without sign} with inequalities $\le$ (resp. $\ge$) instead of equalities.
	
The operator $\L$ appears in different contexts to model nonlocal dispersal. In particular, it is used in population dynamics, where $\int J(x-y)u(y,t)\,{\rm d}y$ represents the population arriving at $x$ at time $t$ from every point in space, while $u(x,t)=\int J(x-y)u(x,t)\,{\rm d}y$ represents the population  leaving the point $x$ at time $t$. Thus, $\L u(x,t)$ is the balance between how much is arriving and how much is leaving $x$ at time $t$. Thus, in the absence of sources or sinks, $\L u=u_t$. If there is absorption, modeled by $-|u|^{p-1}u$, we get~\eqref{eq-problem without sign}; see, for instance, \cite{Fife-2003,Carrillo-Fife-2005,Bates-Zhao-2007,Alfaro-Coville-2017,Coville-2021,Brasseur-Coville-2023}. Equations of this type have also been proposed to model phase transitions \cite{Bates-Chmaj-1999a,Bates-Chmaj-1999b} and  image enhancement \cite{Gilboa-Osher-2008}.

As in the case $\L=\Delta$, the long-time behavior of nonnegative solutions is influenced by two intertwined aspects:
\begin{itemize}
    \item[\textup(i)] the competition between the diffusion and the absorption term;
    \item[\textup(ii)] the asymptotic behaviour of $u_0(x)$ as $|x|\to\infty$.
\end{itemize}

In order to see which mechanism, diffusion or absorption, induces the fastest decay, and hence dominates the evolution, we compare their characteristic decay rates. Solutions to the purely absorbing equation  $u_t=-u^p$ decay like $t^{-1/(p-1)}$. The decay of solutions to the purely diffusive equation $u_t-\L u=0$ depends on the size of the initial datum. When the datum is integrable, the decay is of order $t^{-N/2}$, since in this case solutions to this problem behave for large times like solutions to the \emph{local} heat equation
\begin{equation}\label{eq:local.heat.equation}
    u_t=\a\Delta u,\qquad \a=\frac1{2N}\int J(x)|x|^2\,\textup{d}x
\end{equation}
with the same initial datum; see~\cite{Chasseigne-Chaves-Rossi-2006}. Thus, diffusion should dominate if $\frac1{p-1}<\frac N2$. This is indeed the case: when $p>1+\frac 2N$ and the initial data are integrable (and nonnegative and bounded), then the solution $u$ to problem~\eqref{eq-problem without sign} satisfies
\[
    \lim_{t\to\infty}t^{N/2}\|u(\cdot,t)-M_*\mathcal{U}(\cdot,t)\|_{L^\infty(B_{K\sqrt t})}=0
\]
for every $K>0$, where $\mathcal{U}$ is the fundamental solution of~\eqref{eq:local.heat.equation} and
\[
    M_*=\int u_0(x)\,\textup{d}x-\int_0^\infty\int u^p(x,t)\,\textup{d}x\,\textup{d}t>0
\]
is the asymptotic (nontrivial) mass of $u$; see~\cite{PR}. Notice that, though diffusion is dominant, absorption is still felt: the limit mass, though nontrivial, is not the initial one.

The restriction of this convergence result to the so-called \emph{diffusive scale}, $x\in B_{K\sqrt{t}}$, $K>0$, is not expected to be technical, since, as in the local case (see~\cite{Herraiz-1999}), the behavior at infinity of the initial datum should have a strong influence in the long-time behavior of solutions in other scales. The restriction to such scales will appear in all the results that we describe below.

When the initial data are not integrable, the large time behavior of solutions to the heat equation may be very complex if the initial data have a complex behavior at infinity, even in the local case; see, for instance, \cite{Vazquez-Zuazua-2002}.  Hence, we restrict ourselves to initial data having a power-like decay at infinity,
\begin{equation}\label{eq-condition at infinity-other}
	|x|^\alpha u_0(x)\to A>0\quad\mbox{as }|x|\to\infty.
\end{equation}
When $\alpha \le N$, solutions to the purely diffusive nonlocal equation with such data also behave for large times as the solution to the local heat equation \eqref{eq:local.heat.equation} with the same datum. This can be proved, for instance, with scaling techniques, borrowing ideas from~\cite{terra-wol} to obtain compactness. The large-time behavior for this local problem is given (see, for instance,~\cite{Herraiz-1999}):
\begin{itemize}
    \item If $\alpha<N$, by the unique solution $\mathcal{U}_{A,\alpha}$ to problem~\eqref{eq:local.heat.equation} with initial datum $A|x|^{-\alpha}$. This function has a self-similar form,
        \[
            \mathcal{U}_{A,\alpha}(x,t)=t^{-\alpha/2}F_A(x/t^{1/2}),
        \]
        where $F_A$ is a smooth positive profile such that $\lim_{|\xi|\to\infty}|\xi|^{\alpha} F_A(\xi)=A$.
    \item If $\alpha=N$, by $C_{A,N}\log t\,\mathcal{U}(x,t)$, where $\mathcal{U}$ is the fundamental solution to problem~\eqref{eq:local.heat.equation} and $C_{A,N}$ is a constant that depends only on $A$ and~$N$, .
\end{itemize}
Therefore, diffusion is expected to dominate if $\alpha/2>1/(p-1)$. This expectation was confirmed in~\cite{Terra-Wolanski-PAMS} (see also~\cite{terra-wol}), where it was proved that if $u$ is a solution to \eqref{eq-problem without sign} with an initial datum $u_0$ satisfying~\eqref{eq-condition at infinity-other} with
\[
    \alpha>\frac2{p-1},
\]
then, for any $K>0$,
\[
    \begin{array}{ll}
        \displaystyle\lim_{t\to\infty}t^{\alpha/2}\|u(\cdot,t)-\mathcal{U}_{A,\alpha}(\cdot,t)\|_{L^\infty(B_{K\sqrt t})}=0\quad&\text{if }\alpha\in(0,N),\\
        \displaystyle\lim_{t\to\infty}\frac{t^{N/2}}{\log t}\|u(\cdot,t)-C_{A,N}\log t\,\mathcal{U}(\cdot,t)\|_{L^\infty(B_{K\sqrt t})}=0\quad&\text{if }\alpha=N.
    \end{array}
\]
Notice that in this case absorption has no effect in the large-time behaviour.

Let us take a look now to the region of parameters where we expect the decay rate of solutions to be that corresponding to absorption, $t^{-1/(p-1)}$. This region consists of two parts: $p\in(1,1+\frac2N)$ if the initial data are integrable, and $\alpha<2/(p-1)$ when they are not and have the decay behavior~\eqref{eq-condition at infinity-other}.

Though the absorption mechanism is dominant in this region of parameters, diffusion is still there, and may play a more or less important role, depending on the behavior of the initial datum $u_0$ as $|x|\to\infty$. If $u_0$ has a fast decay, we expect bigger gradients, and hence diffusion will be more significant, while if the decay is slow, solutions may become \lq\lq flat'' in diffusive regions, and then decay there as a solution to the purely absorbing problem.

To identify the expected limits, depending on the behavior of the datum at infinity, we scale the solution, taking into account that we expect to have the decay rate corresponding to absorption. Thus, for each $\lambda>0$ we define
\begin{equation}\label{eq:asorption.scaling}
    \ul(x,t)=\la^{\frac2{p-1}} u(\la x,\la^2 t).
\end{equation}
If the family $\{\ul\}$ converges uniformly in compact sets as $\lambda\to\infty$ towards some function~$U$, taking $t=1$ and renaming $\lambda^2$ as $t$, we will get the convergence of $u$ towards $U$ in the diffusive scale.  But, which function is $U$? Let us explain heuristically how to characterize~it.

The scaled function $\ul$ is a solution to
\begin{equation*}\label{eq-problem laplacian lambda}
        u_t-\L_\la u=-u^{p}\quad\mbox{in }\R^N\times(0,\infty),\qquad u(x,0)=\la^{\frac2{p-1}} u_0(\la x)\quad\mbox{in }\R^N,
\end{equation*}
for a scaled operator $\L_\la$ given by
\begin{equation}\label{eq:scaled.operator}
    \L_\la u(x)=\la^2\int J_\la(x-y)\big(u(y)-u(x)\big)\,{\rm d}y,\qquad\text{with }J_\la(z)=\la^N J(\la z).
\end{equation}
The operators $\L_\la$ converge, in a certain weak sense, to $\a \Delta$ as $\lambda\to\infty$; see, for instance,~\cite{Rossi-book} and the references therein. Therefore, the limit function $U$ is expected to be a solution to the semilinear local heat equation
\begin{equation}\label{eq:local.heat.absorption}
    U_t-\a \Delta U=-U^p\quad\mbox{in }\R^N\times(0,\infty), \qquad \a=\frac1{2N}\int J(z)|z|^2\,\textup{d}z.
\end{equation}
On the other hand, we expect the limit function $U$ to be related to the initial datum of the original problem, so that the trace of $U$ is the limit of the family $\{u_{0,\lambda}\}$,
\[
    U(x,0^+):=\lim_{t\to0^+}U(x,t)=\lim\limits_{\lambda\to\infty}u_{0,\lambda}(x)=\lim\limits_{\lambda\to\infty}\la^{\frac2{p-1}}u_0(\lambda x).
\]
If this is the case and $u_0$ decays slowly, then
\begin{equation}\label{eq:slow.decay}
    U(x,0^+)=\lim\limits_{\lambda\to\infty}\la^{\frac2{p-1}}u_0(\lambda  x)=|x|^{-\frac2{p-1}}\lim_{|\lambda x|\to\infty}|\lambda x|^{\frac2{p-1}}u_0(\lambda x)=\infty\quad \textup{if }x\neq0.
\end{equation}
This is compatible with $U$ begin the \emph{flat} solution  $\mathcal{F}(t)=\big((p-1)t\big)^{-\frac1{p-1}}$ to \eqref{eq:local.heat.absorption}, and, indeed, if $u_0$ satisfies~\eqref{eq:slow.decay}, then, for any $K>0$,
\[
    \lim_{t\to\infty}t^{\frac1{p-1}}\|u(\cdot,t)-\mathcal{F}(t)\|_{L^\infty(B_{K\sqrt t})}=0;
\]
see~\cite{STW}. Notice that for initial data having the power-like behavior at infinity~\eqref{eq-condition at infinity-other}, condition~\eqref{eq:slow.decay} is equivalent to
\[
    \alpha<\frac2{p-1}.
\]
This covers the whole absorption dominated region of parameters when the initial data are not integrable.

There is a situation in the absorption dominated region that has not been treated yet,  namely, the case $p\in(1,1+\frac2N)$ with integrable initial datum such that $|x|^{\frac2{p-1}} u_0(x)\to A\ge0$ as $|x|\to\infty$. For $u_0$ satisfying~\eqref{eq-condition at infinity-other}, this would correspond to $\alpha\ge\frac2{p-1}>N$. The aim of this paper is to fill in this gap.

When the limit value $A$ is 0, then $\lim\limits_{\lambda\to\infty}\la^{\frac2{p-1}}u_0(\la x)=0$ for all $x\neq0$, while
$$
    \lim\limits_{\lambda\to\infty}\int \la^{\frac2{p-1}}u_0(\la x)\,\textup{d}x=\la^{\frac2{p-1}-N}\int u_0(\xi)\,\textup{d}\xi=+\infty,
$$
since $p<1+\frac2N$. Thus, the expected limit  $U$ should satisfy
\begin{equation}\label{eq:conditions.VSS}
    U(x,t)\to0 \quad \textup{as }t\to0^+\textup{ for all } x\neq0,\quad \int U(x,t)\,{\rm d}x\to\infty \quad \textup{as }t\to0^+.
\end{equation}
There is a unique nonnegative solution $V$ to the semilinear local heat equation~\eqref{eq:local.heat.absorption} with this very singular behavior as $t\to0^+$. This nonnegative \emph{very singular solution} to~\eqref{eq:local.heat.absorption} should then give the large-time behavior in this case.   We confirm this expectation in Theorem~\ref{teo-main}, where we prove that for every $K>0$
\[
    t^{\frac1{p-1}}\|u(\cdot,t)-V(\cdot,t)\|_{L^\infty(B_{K\sqrt t})}\to0\quad\mbox{as }t\to\infty.
\]

When $A>0$, the initial datum of the expected limit is $U(x,0)=A|x|^{-\frac2{p-1}}$, and we will prove in Theorem~\ref{teo-main2} that, for every $K>0$,
\begin{equation}\label{eq:convergence.critical.decay.initial.datum}
    t^{\frac1{p-1}}\|u(\cdot,t)-U_A(\cdot,t)\|_{L^\infty(B_{K\sqrt t})}\to0\quad\mbox{as }t\to\infty,
\end{equation}
where $U_A$ is the unique solution to the semilinear local heat equation~\eqref{eq:local.heat.absorption} such that
\begin{equation}\label{eq:initial.datum.profile.critical.decay}
    U(x,t)\to A|x|^{-\frac2{p-1}}\textup{ for all } x\neq0\quad \mbox{as } t\to0^+.
\end{equation}

Our proofs are based on the scaling argument mentioned above. The most important technical  difficulties appear,   first, when trying to prove that the family  $\{\ul\}$ is uniformly bounded in $t\ge t_0>0$, and then, once this is established,  showing that the family $\{\ul\}$ is precompact on compact sets.   The second difficulty stems mainly from: (i)  the fact that the functions $\ul$ are solutions of equations involving different operators; and (ii) the lack of a regularizing effect of the diffusion operators themselves. This difficulty was overcome in \cite{terra-wol} with arguments that can be used in our present situation.

In comparison with the available large-time asymptotic results, the main difference is that in this case the initial condition is not attained in the sense of finite measures, so that the characterization of the limit solution needs new and different ideas.

\noindent\emph{Borderline cases.}  If $u_0$ satisfies~\eqref{eq-condition at infinity-other} with
\[
    \alpha=\frac{2}{p-1}< N,
\]
both decay mechanisms, diffusion and absorption, are  balanced. But we may still proceed as in the case $\alpha=\frac{2}{p-1}>N$ considered in the present paper, performing the scaling~\eqref{eq:asorption.scaling} of the solution to~\eqref{eq-problem without sign} to obtain a similar result. Indeed, the large-time behavior is given by~\eqref{eq:convergence.critical.decay.initial.datum} for any $K>0$, with $1/(p-1)=\alpha/2$ and $U_A$ the unique solution to the semilinear local heat equation~\eqref{eq:local.heat.absorption} satisfying~\eqref{eq:initial.datum.profile.critical.decay};  see~\cite{terra-wol}. The difference with respect to the case in which $\alpha=\frac{2}{p-1}>N$ is that now the absorption decay rate coincides with the diffusion one.  

When the initial datum is integrable, $u_0\in L^1(\mathbb{R}^N)$, the borderline balanced case corresponds to $p=1+\frac2N$. However, the decay rate is not the common one for diffusion and absorption, $t^{-\frac N2}$, but faster,
\[
    \lim_{t\to\infty}t^{N/2}u(\cdot ,t)\to0\quad\mbox{uniformly on compact sets},
\]
due to the combined effects of both mechanisms; see~\cite{terra-wol}. Hence, the decay rate needs some correction, as compared to the diffusion/absorption one, $t^{-\frac N2}$.

 There is also a doubly critical case, namely $p=1+\frac2N$ with initial data satisfying~\eqref{eq-condition at infinity-other} with $\alpha=N$, so that $\alpha=\frac2{p-1}$. Notice that for these borderline values of the parameters diffusion and absorption are not balanced, the decay due to the latter being faster ($t^{-\frac1{p-1}}$ against $t^{-\frac N2}\log t$). In view of the results for the local case $\mathcal{L}=\Delta$ (see~\cite{Herraiz-1999}), we expect the same behavior as for the cases $\alpha=\frac2{p-1}\neq N$ mentioned above. However, the proofs might be more involved, due to the doubly critical character of the problem.

These two remaining cases are expected to require new ideas, and will be considered elsewhere.  

\subsection*{Organization of the paper}
In Section~2 we first prove existence and uniqueness of a local solution  when $u_0\in L^1(\mathbb{R}^N)\cap  L^\infty(\mathbb{R}^N)$, allowing for sign changes. Next, we obtain a comparison result that implies, in particular, that if the initial datum is nonnegative, then the solution is nonnegative and exists globally. In Section~3 we obtain space and time decay estimates for the solution that are scaling invariant. Section~4 is devoted to the proof of our main result. After choosing the appropriate scaling $u_\lambda$ for the absorption dominated cases, we use the time estimate from the previous section to show that the family $\{\ul\}$ is precompact in compact subsets of~$t>0$, so that it converges under subsequences to a function~$U$. Then, thanks to the space estimate, we identify the limit function~$U$. It is here where the main technical difficulties of the paper arise. Finally, we translate the convergence as $\lambda\to\infty$ in our large-time behavior results for solutions to~\eqref{eq-problem without sign}.

\section{Existence  and comparison }
\setcounter{equation}{0}

 The goal of this section is to prove that problem \eqref{eq-problem without sign} has a unique global in time nonnegative  solution if the initial datum is nonnegative and belongs to $L^1(\mathbb{R}^N)\cap  L^\infty(\mathbb{R}^N)$.  
We begin by proving  existence and uniqueness of a \emph{local in time solution} for such initial data, that is a function $u\in C([0,T];L^1(\R^N))\cap  C^1(0,T;L^1(\R^N))\cap  L^\infty(0,T;L^\infty(\R^N))$ for some $T>0$   that satisfies
\begin{equation}\label{eq:local.solution}
    \displaystyle u_t-\L u=-|u|^{p-1}u\quad\mbox{in }\R^N\times(0,T),\qquad u(x,0)=u_0(x)\quad\mbox{in }\R^N.
\end{equation}
Sign changes are allowed.   Then, we prove a comparison principle for sub- and super-solutions  (in their common existence time)   that implies, in particular, that if $u_0\ge0$ then $u\ge0$. This in turn implies, if $u_0\ge0$, that $u\le u_\L$ for $t\in[0,T]$,  where $u_\L$ stands for the solution to
\begin{equation}\label{eq-homogeneous}
	u_t-\L u=0\quad\mbox{in }\R^N\times(0,\infty),\qquad u(x,0)=u_0(x)\quad\mbox{in }\R^N.
\end{equation}	
Hence, for every $t\in[0,T]$,
\begin{equation}\label{eq:L1.Linfty.estimates}
    \|u(\cdot,t)\|_{L^1(\R^N)}\le \|u_0\|_{L^1(\R^N)},\qquad \|u(\cdot,t)\|_{L^\infty(\R^N)}\le \|u_0\|_{L^\infty(\R^N)}.
\end{equation}
Since the time step in the construction of the solution depends only on $\|u_0\|_{L^\infty(\R^N)}$, we  will conclude that the solution exists globally.
	
We start with the proof of local existence, that uses a fixed point argument.
\begin{teo}\label{teo-existence}
    Let $p>1$ and $u_0\in L^1(\R^N)\cap L^\infty(\R^N)$. There exists $T>0$ and a unique function $u\in C([0,T];L^1(\R^N))\cap  C^1(0,T;L^1(\R^N))\cap  L^\infty(0,T;L^\infty(\R^N))$ that satisfies~\eqref{eq:local.solution}.
\end{teo}

\begin{proof}
Let
\begin{align*}
    K=\{v&\in L^\infty(0,T;L^1(\R^N))\cap  L^\infty(0,T;L^\infty(\R^N))\mbox{ such that }\|v(\cdot,t)\|_{L^1(\R^N)}\le 2\|u_0\|_{L^1(\R^N)}\\
    &\mbox{ and }\|v(\cdot,t)\|_{L^\infty(\R^N)}\le 2\|u_0\|_{L^\infty(\R^N)}\mbox{ for every }t\in (0,T)\},
\end{align*}
and for each $v\in K$ let
\begin{align*}
    \T v(x,t)&=e^{-t}u_0(x)+\int_0^t\int e^{-(t-s)}J(x-y)v(y,s)\,{\rm d}y{\rm d}s\\
    &\quad-\int_0^te^{-(t-s)}|v(x,s)|^{p-1}v(x,s)\,{\rm d}s.
\end{align*}

It is easy to check that if $u$ is a fixed point of $\T$ in $K$, then $u\in C([0,T];L^1(\R^N))$, and once this is proved, it is immediate that $u\in C^1(0,T;L^1(\R^N))$ and satisfies~\eqref{eq:local.solution}. Therefore, $u$ is a solution to~\eqref{eq:local.solution} if and only if it is a fixed point of $\T$ in $K$. Hence, let us prove that $\T$ has a unique fixed point in $K$ if we take $T$ small enough. We will do this by means of Banach's fixed-point theorem.  
	
Let us see  first that $\T:K\to K$ if $T$ is small, how small depending only on  $\|u_0\|_{L^\infty(\R^N)}$. Indeed, on the one hand,
\begin{align*}
    \|\T v(\cdot,t)\|_{L^1(\R^N)}&\le e^{-t}\|u_0\|_{L^1(\R^N)}+\int_0^t e^{-(t-s)}\|v(\cdot,s)\|_{L^1(\R^N)}\,{\rm d}s\\
	&\hskip1cm+\int_0^te^{-(t-s)}\|v(\cdot,s)\|_{L^\infty(\R^N)}^{p-1}\|v(\cdot,s)\|_{L^1(\R^N)}\,{\rm d}s\\
	&\le (1+2t+2^p\|u_0\|_{L^\infty(\R^N)}^{p-1}t)\|u_0\|_{L^1(\R^N)}\le 2\|u_0\|_{L^1(\R^N)}
\end{align*}
if $t\le T$, for some $T$ depending only on  $\|u_0\|_{L^\infty(\R^N)}$. Analogously, with the same $T$, for $t\in (0,T)$ we have
\[
    \|\T v(\cdot,t)\|_{L^\infty(\R^N)}\le (1+2t+2^p\|u_0\|_{L^\infty(\R^N)}^{p-1}t)\|u_0\|_{L^\infty(\R^N)}\le 2\|u_0\|_{L^\infty(\R^N)}.
\]

In order to see that $\T$ it is a contraction in $K$ with the norm
\[
    \vvvert v \vvvert=\max\{\|v\|_{L^\infty(0,T;L^1(\R^N))},\|v\|_{L^\infty(0,T;L^\infty(\R^N))}\},
\]
we observe that, if $v_1$ and $v_2$ belong to $K$, for every $0<s<t\le T$,
\[
    \big||v_1(x,s)|^{p-1}v_1(x,s)-|v_2(x,s)|^{p-1}v_2(x,s)\big|\le p(2\|u_0\|_{L^\infty(\R^N)})^{p-1}|v_1(x,s)-v_2(x,s)|,
\]
so that
\begin{align*}
    \|\T v_1-\T v_2\|_{L^\infty(0,T;L^1(\R^N))}&\le (1+p(2\|u_0\|_{L^\infty(\R^N)})^{p-1})T\|v_1-v_2\|_{L^\infty(0,T;L^1(\R^N))}\\
    &\le \frac12\|v_1-v_2\|_{L^\infty(0,T;L^1(\R^N))}
\end{align*}
if $T$ is small enough, how small depending only on  $\|u_0\|_{L^\infty(\R^N)}$. Analogously, for this same~$T$,
\begin{align*}
    \|\T v_1-\T v_2\|_{L^\infty(0,T;L^\infty(\R^N))}&\le (1+p(2\|u_0\|_{L^\infty(\R^N)})^{p-1})T\|v_1-v_2\|_{L^\infty(0,T;L^\infty(\R^N))}\\
    &\le\frac12\|v_1-v_2\|_{L^\infty(0,T;L^\infty(\R^N))}.
\end{align*}
Hence, $\T$ is a contraction in $K$ if $T$ is small, how small depending only on $\|u_0\|_{L^\infty(\R^N)}$. Therefore, $\T$ has a  unique fixed point  $u\in K$, which is the unique solution to~\eqref{eq-problem without sign} in~$\R^N\times(0,T)$.  
\end{proof}
	
Now, we prove a comparison result.		
\begin{prop}\label{prop-comparison}
    Let $u_1$ be a subsolution to~\eqref{eq-problem without sign} and $u_2$ a supersolution to~\eqref{eq-problem without sign} in $\R^N\times[0,T]$	
    such that $u_1(x,0)\le u_2(x,0)$ in $\R^N$. Then, $u_1(x,t)\le u_2(x,t)$ in $\R^N\times[0,T]$.
\end{prop}
\begin{proof}
Let $w=u_1-u_2$. Then,
\[
    w_t(x,t)\le\L w(x,t)-(|u_1(x,t)|^{p-1}u_1(x,t)-|u_2(x,t)|^{p-1}u_2(x,t)).
\]
We multiply this inequality  by $\operatorname{sgn}(w)$ (where $\operatorname{sgn}(s)=1$ if $s>0$ and $\operatorname{sgn}(s)=0$ if $s\le0$) and integrate in $\R^N$. Hence, as $w_t \operatorname{sgn}(w)=(w^+)_t$, with $w^+=\max\{w,0\}$,
\begin{align*}
    \int_{\R^N} (w^+)_t&(x,t)\,{\rm d}x\le\int_{\R^N}\int_{\R^N}J(x-y)\big(w(y,t)-w(x,t)\big)\operatorname{sgn}(w(x,t))\,{\rm d}y{\rm d}x\\
    &\hskip.5cm-\int_{\R^N}(|u_1(x,t)|^{p-1}u_1(x,t)-|u_2(x,t)|^{p-1}u_2(x,t))\operatorname{sgn}(u_1(x,t)-u_2(x,t))\,{\rm d}x\\
    &\le\int_{\R^N}\int_{\R^N}J(x-y)\big(w(y,t)-w(x,t)\big)\operatorname{sgn} (w(x,t))\,{\rm d}y{\rm d}x\\
    &=-\frac12\int_{\R^N}\int_{\R^N}J(x-y)\big(w(y,t)-w(x,t)\big)\big(\operatorname{sgn}(w(y,t))-\operatorname{sgn}(w(x,t)\big)\,{\rm d}y{\rm d}x\\
    &\le 0.
\end{align*}
Integrating between $0$ and $t$, $t\in (0,T]$, we deduce that $\displaystyle\int_{\R^N}w^+(x,t)\,{\rm d}x=0$. Hence, $w(\cdot,t)\le0$ for all $t\in(0,T]$, that is, $u_1\le u_2$ in $\R^N\times(0,T]$.
\end{proof}
As a corollary we get positivity if $u_0\ge0$, since $v\equiv0$ is a solution to the equation.
\begin{coro}\label{coro-nonnegative}
    Let $u$ be the solution constructed in Theorem \ref{teo-existence}. Let us assume that $u_0\ge0$. Then, $u\ge0$ as well.
\end{coro}
Thanks to positivity, $u$ is a sub-solution to \eqref{eq-homogeneous}. Therefore, $u\le u_\L$ and, in particular, the estimates~\eqref{eq:L1.Linfty.estimates} hold for all $t\in[0,T]$.

As the existence time $T$ in Theorem \ref{teo-existence} depends only on  the $L^\infty$ norm of the initial datum, we can extend the solution step by step and get that the solution exists globally.
\begin{coro}\label{coro-global existence}
    Let $0\le u_0\in L^1(\R^N)\cap L^\infty(\R^N)$. Then, there exists a unique maximal  solution to \eqref{eq-problem without sign}. This solution is global.
\end{coro} 	

\begin{rem}\label{rem-comparison}
\textup{(i)}
The comparison principle also holds if $u_1$ and $u_2$ are respectively a sub and a super-solution to the nonlocal equation in~\eqref{eq-problem without sign} in $|x|>a>0$, $0<t<T$, $u_1(x,0)\le u_2(x,0)$ in $|x|>a>0$ and, in addition, $u_1(x,t)\le u_2(x,t)$ if $|x|\le a$, $0<t<T$.
 		
To prove this we proceed as in the proof of Proposition~\ref{prop-comparison}. This time, the equation holds in $|x|>a$ and not in $\R^N$. But, as $\operatorname{sgn}(w(x,t))=0$ if $|x|<a$, we have that
\begin{align*}
    \int_{{|x|>a}}(w^+)_t(x,t)\,{\rm d}x&\le\int_{{|x|>a}}\int_{\R^N}J(x-y)\big(w(y,t)-w(x,t)\big)\operatorname{sgn}(w(x,t))\,{\rm d}y{\rm d}x\\
	&=\int_{\R^N}\int_{\R^N}J(x-y)\big(w(y,t)-w(x,t)\big)\operatorname{sgn}(w(x,t))\,{\rm d}y{\rm d}x.
\end{align*}
From here on, the proof follows as in Proposition~\ref{prop-comparison}.	
 		
\noindent\textup{(ii)} An analogous result holds in $B_a\times(0,T)$ if $u_1$ and $u_2$ are respectively a sub and a super-solution in $B_a\times(0,T)$, $u_1(x,0)\le u_2(x,0)$ in $B_a$ and, in addition,  $u_1(x,t)\le u_2(x,t)$ if $a\le|x|\le a+d$, $0<t<T$.
\end{rem}

As a corollary of the proof of Proposition \ref{prop-comparison} and, with the idea just stated in the remark, we have the following result that will be used at the end of the paper.
\begin{prop}\label{prop-comparison 2}
    Let $\Omega\subset\mathbb{R}^N$ bounded. Let $Z\in C([0,T];L^1(\R^N))\cap C^1(0,T;L^1(\R^N))\cap L^\infty(0,T;L^\infty(\R^N))$ be such that
    \begin{equation*}\label{eq-comparison 2}
 		\begin{cases}
 			Z_t\le \L Z-F&\textup{in }\Omega\times(0,T),\\
 			Z\le 0&\textup{in }(B_d(\Omega)\setminus\Omega)\times(0,T),\\
 			Z\le 0&\textup{in }\Omega\times\{0\}.
		\end{cases}
    \end{equation*}
    with $F(x,t)\operatorname{sgn}(Z(x,t))\ge0$. Then, $Z\le 0$ in $\Omega\times(0,T)$.
\end{prop}	

\section{Time and space decay}
\setcounter{equation}{0}

In this section we prove that if $0\le u_0\in L^\infty(\R^N)$, $p\in(1,1+\frac2N)$ and $|x|^{\frac2{p-1}}u_0(x)\to A\ge0$ as $ |x|\to \infty$, then there are positive constants $C_0$ (depending of $p$, $J$ and $u_0$) and $C_p$ (depending only on $p$) such that
\begin{equation}\label{eq:invariant.time and space decay}
    u(x,t)\le C_0|x|^{-\frac2{p-1}}, \qquad u(x,t)\le C_p\, t^{-\frac1{p-1}}.
\end{equation}
Note that this size estimates are invariant under the scaling~\eqref{eq:asorption.scaling}.

In order to obtain the estimate for the spatial decay, we construct a stationary supersolution.
\begin{prop}\label{prop-supersolution}
    Let $\phi(x)=(C_1+C_2|x|^2)^{-\frac1{p-1}}$, $C_1,C_2>0$. If $C_2$ is small enough, then $\L\phi(x)\le\phi^p(x)$ for $|x|\ge2d$. If in addition $C_1\le C_2$ and $C_2$ is small enough, then $\phi\ge u_0$ in $\R^N$ and $\phi(x)\ge\|u_0\|_{L^\infty(\R^N)}\ge u(x,t)$ for all $t>0$ if $|x|\le 2d$.
\end{prop}

\begin{proof}
As $J$ is radially symmetric, a Taylor's expansion shows that
\begin{align*}
    \L\phi(x)&=\int J(x-y)\big(\phi(y)-\phi(x)\big)\,{\rm d}y\\
    &=\frac12\sum_{i,j=1}^N\int J(x-y)(y_i-x_i)(y_j-x_j)\Big(\int_0^1\phi_{x_ix_j}(x+s(y-x))\,{\rm d}s\Big)\,{\rm d}y.
\end{align*}
To estimate the right-hand side we observe that, since $C_2>0$,
\begin{align*}
    \phi_{x_ix_j}(x)&=\frac{4p}{(p-1)^2}C_2^2x_ix_j(C_1+C_2|x|^2)^{-\frac p{p-1}-1}-\delta_{ij}\frac{2C_2}{p-1}(C_1+C_2|x|^2)^{-\frac p{p-1}}\\
    &\le \frac{4p}{(p-1)^2}C_2\frac{C_2|x|^2}{C_1+C_2|x|^2}(C_1+C_2|x|^2)^{-\frac p{p-1}}\le\frac{4p}{(p-1)^2}C_2(C_1+C_2|x|^2)^{-\frac p{p-1}}.
\end{align*}
Besides, if $|x|\ge2d$ and $|x-y|<d$,  then $|x+s(y-x)|\ge \frac{|x|}2$ for every $0<s<1$. Hence, using again the radial symmetry of the kernel, if $|x|\ge2d$ we have
\begin{align*}
    \L\phi(x)& \le \frac{2p}{(p-1)^2}C_2\Big(\int J(z)|z|^2\,{\rm d}z\Big) \Big(C_1+\frac{C_2}4|x|^2\Big)^{-\frac p{p-1}}\\
    &\le 4^{\frac p{p-1}}\frac{2p}{(p-1)^2}C_2\Big(\int J(z)|z|^2\,{\rm d}z\Big)\ \phi(x)^p\le \phi(x)^p
\end{align*}
if $C_2$ is small enough, how small depending only on $p$ and $J$.

Let now $B>\max\{1,2d\}$ be such that $|x|^{\frac2{p-1}}u_0(x)\le A+1$ if $|x|>B$, and let $ C_2$ be sufficiently small so that $(2C_2)^{-\frac1{p-1}}\ge A+1$. Let $C_1\le C_2$. Then
\[
    \phi(x)\ge (C_2(1+|x|))^{-\frac1{p-1}}\ge (2C_2)^{-\frac1{p-1}}>A+1\ge |x|^{\frac2{p-1}}u_0(x)\ge u_0(x)\quad\textup{if }|x|>B.
\]
Taking $C_2$ even smaller, if needed,  so that $(C_1+C_2B^2)^{-\frac1{p-1}}\ge\|u_0\|_{L^\infty(\R^N)}$, then
\[
    \phi(x)\ge(C_1+C_2B^2)^{-\frac1{p-1}}\ge\|u_0\|_{L^\infty(\R^N)}\quad\textup{if }|x|\le B.
\]
Summarizing, if $C_1\le C_2$ and $C_2$ is small enough, then, on the one hand, $\phi(x)\ge u_0(x)$ for all $x\in\mathbb{R}^N$, and on the other hand, as $B>2d$, $\phi(x)\ge\|u_0\|_{L^\infty(\R^N)}\ge u(x,t)$ if $|x|\le 2d$ for all $t>0$.
\end{proof}

As a corollary we obtain a spatial decay estimate for $u$ by comparison.
\begin{coro}\label{coro-space decay}
    There is a constant $C_0$ such that
    \[
        0\le |x|^{\frac 2{p-1}}u(x,t)\le(1+|x|)^{\frac 2{p-1}}u(x,t)\le C_0\quad\text{in }\R^N\times(0,\infty).
    \]
\end{coro}
\begin{proof}
By Proposition \ref{prop-supersolution}, $\phi$ is a  supersolution in $|x|>2d$ to the nonlocal equation in~\eqref{eq-problem without sign}. Moreover, $\phi\ge u_0$ in $\R^N$ and $\phi(x)\ge u(x,t)$ in $|x|\le 2d$, $t>0$. Hence, by comparison (see Remark~\ref{rem-comparison}), we have that $u(x,t)\le \phi(x)$ for all $x\in\mathbb{R}^N$, $t>0$.
\end{proof}

Now, we obtain the time decay.
\begin{prop}\label{prop-time decay}
    There exists a constant $C_p$ depending only on $p$ such that
    \[
            u(x,t)\le C_p t^{-\frac1{p-1}}.
    \]
\end{prop}
\begin{proof}
Given any $t_0>0$, the function $g(t)=C_p(t+t_0)^{-\frac1{p-1}}$, where $C_p=(p-1)^{-\frac1{p-1}}$, is a solution to the ODE $g'(t)=-g(t)^p$. If $t_0$ is small enough,  $g(0)\ge \|u_0\|_{L^\infty(\R^N)}$. On the other hand, by Corollary~\ref{coro-space decay}, given $T>0$ there is some $\beta=\beta(T,p)$ such that $g(t)\ge u(x,t)$ if $t\in(0,T)$ and $|x|\ge \beta$. Hence, by comparison (see Remark~\ref{rem-comparison}), $g(t)\ge u(x,t)$ for all~$t\in(0,T)$ and $x\in\mathbb{R}^N$. Letting $t_0\to 0^+$, we deduce that $u(x,t)\le C_p t^{-\frac1{p-1}}$ for all $t\in (0,T)$ and $x\in\mathbb{R}^N$. As $T>0$ is arbitrary, the result is proved.	
\end{proof}

\section{Large-time behavior}
\setcounter{equation}{0}

 We now prove our large-time behavior results. As announced in the Introduction, the first step is to scale the solution $u$ to~\eqref{eq-problem without sign}. The right scaling, given by~\eqref{eq:asorption.scaling}, is suggested by the time-decay estimate for $u$ in the previous section. We then obtain compactness, that gives the existence of convergent subsequences of the family $\{u_\lambda\}$. The third step, the most involved, is to identify uniquely the possible limits. The fourth and last step is to translate the convergence in $\lambda$ of the functions~$u_\lambda$ into large-time results for the function~$u$.  

\subsection{Scaling}

 The first step is to choose the right scaling. Since we want to keep the absorption term and the diffusive character of the equation, we take
\begin{equation*}\label{eq-rescaled}
    \ul(x,t)=\la^{\frac2{p-1}}u(\la x,\la^2t), \quad \la>0.
\end{equation*}
The key point is that, thanks to the estimates~\eqref{eq:invariant.time and space decay} obtained in the previous section, we know that
\[
     \ul(x,t)\le C_pt^{-\frac1{p-1}},\qquad \ul(x,t)\le C_0|x|^{-\frac2{p-1}}.
\]
In particular, the family $\{u_\lambda\}$ is locally bounded in $\R^N\times(0,\infty)$.

\subsection{Compactness}

Once we know that the family $\{\ul\}$ is locally bounded in $\R^N\times(0,\infty)$, its precompactness is obtained with \emph{the same computations} as in~\cite[Section 3]{terra-wol}. The functions $f$ and $F$ in that paper correspond here to $f(\lambda)=\la^{\frac2{p-1}}$ and $F(\la)=1$. Hence, we have the following result.
\begin{prop}\label{prop-compactness}
    The family $\{\ul\}$ has a subsequence $\{u_{\lambda_k}\}$ that converges uniformly on compact subsets of $\R^N\times(0,\infty)$ to a function $U$ satisfying the semilinear local heat equation~\eqref{eq:local.heat.absorption}.
\end{prop}

\subsection{Identification of the limit function $U$}

The next step is to characterize the possible limit functions $U$. We will prove  that there is only one possible limit, so that convergence is not restricted to a subsequence. 

We consider separately the cases $A>0$ and $A=0$. When $A=0$ we prove that $U=V$, where $V$ is the unique nonnegative very singular solution to the semilinear heat equation~\eqref{eq:local.heat.absorption}, that is,  the unique nonnegative solution to this equation satisfying~\eqref{eq:conditions.VSS}. When $A>0$ we prove that $U=U_A$, where $U_A$ is the unique solution to~\eqref{eq:local.heat.absorption} satisfying~\eqref{eq:initial.datum.profile.critical.decay}, that is, with $A|x|^{-\alpha}$ as initial trace. 

\begin{rem}
    The existence of solutions $V$ and $U_A$, $A>0$, to~\eqref{eq:local.heat.absorption} with the desired behaviour as $t\to0^+$ had already been proved in~\cite{BPT}. A uniqueness proof for $U_A$ can be found in \cite{K1}. As for the very singular solution $V$, uniqueness was only known within the class of \emph{radially symmetric self-similar} functions; see~\cite{BPT}.  This is not enough for our purposes, since we cannot guarantee \emph{a priori} that the limit $V$ has these two properties. However, starting from the uniqueness result in the restricted class, and repeating the arguments used in~\cite{K2} to deal with the same issue for the fast diffusion equation, the uniqueness of a nonnegative very singular solution to~\eqref{eq:local.heat.absorption} without further assumptions follows.
\end{rem}

We start with the case $A=0$. We already know that all limit functions $U$ are solutions to the semilinear local heat equation with absorption~\eqref{eq:local.heat.absorption}. The next step is then to prove that they have the very singular behaviour~\eqref{eq:conditions.VSS} as $t\to0^+$. We start by showing that they go to 0 as $t\to0^+$ for all $x\neq0$.  

\begin{prop}\label{prop-U to 0}
    Let $u_0$ satisfy \eqref{eq-condition at infinity} with $A=0$. Let $U=\lim_{j\to\infty}u_{\la_j}$ with $\la_j\to\infty$. Then, $U(x,t)\to0$ as $t\to0^+$ if $x\neq0$.
\end{prop}

\begin{proof} Let $x_0\neq0$, $a<\frac{|x_0|}2$ and $\la>\frac{4d}{|x_0|}$. Let $B$ be the ball centered at $x_0$ with radius $a$ and $\ep\le \psi\in C^{2,\alpha}(\bar B)$ such that $\psi\equiv\ep$ in $B_{a/4}(x_0)$, $\psi=M$ in $B\setminus B_{a/2}(x_0)$. Here $M\ge u_\la(x,t)$ if $|x|>\frac{|x_0|}4$ (recall that $u_\la\le C_0|x|^{-\frac2{p-1}}$ for every $\la>0$).	Let $\ep'=\ep\big(\frac{|x_0|}2\big)^{\frac2{p-1}}$. By~\eqref{eq-condition at infinity}, $u_{\la}(x,0)\le\ep'|x|^{-\frac2{p-1}}\le \ep$ in $B$ if $\la$ is large enough,  so that $u_{\la}(x,0)\le \psi(x)$ in $B$.
	
Let now $v_\la$ be the solution to
\begin{equation}\label{eq-barrier in B}
    \begin{cases}
		v_t-\L_\la v=0&\mbox{in }B\times(0,\infty),\\
		v=M&\mbox{in }(\R^N\setminus B)\times(0,\infty),\\
		v(x,0)=\psi(x)&\mbox{in }B.
	\end{cases}
\end{equation}
Then, as $v_{\la}$ is a supersolution to \eqref{eq-problem without sign}, $u_{\la}\le v_{\la}$ in $B\times(0,\infty)$. But, as proved in~\cite{CER}, $v_{\la}\to V$ as $\la\to\infty$, where $V$ is the solution to
\begin{equation}\label{eq-V}
    \begin{cases}
 		V_t-\a\Delta V=0&\mbox{in }B\times(0,\infty),\\
 		V=M&\mbox{in }\partial B\times(0,\infty),\\
 		V(x,0)=\psi(x)&\mbox{in }B.
    \end{cases}
\end{equation}
Hence, $U\le V$ in $B\times(0,\infty)$, whence $\limsup_{t\to0^+}U(x_0,t)\le \psi(x_0)=\ep$. As $U\ge0$ and $\ep$ is arbitrary, $\lim_{t\to0^+}U(x_0,t)=0$.
\end{proof}

Let us prove now that $\lim_{t\to0^+}\int U(x,t)\,{\rm d}x=\infty$. To this end, we will find a bound from below for $\int \ul(x,t)\,{\rm d}x$. From the definition of $\ul$ we see that
\[
    \int\ul(x,t)\,{\rm d}x=\la^{\frac2{p-1}-N}\int u(x,\la^2t)\,{\rm d}x.
\]
To find a lower bound for this function, we construct a subsolution to $u$. So, let $L>0$ and $v$ the solution to
\begin{equation}\label{eq-subsolution to bound integral}
    \begin{cases}
        v_t-\L v=-v^p\quad&\mbox{in }\R^N\times(0,\infty),\\
        v(x,0)= L{\la^{N-\frac2{p-1}}}h(x)\quad&\mbox{in }\R^N,
    \end{cases}
\end{equation}
where
$$
    h(x)=u_0(x)\|u_0\|_{L^1(B_a)}^{-1}\chi_{B_a},
$$
with $a$ large so that $\|u_0\|_{L^1(B_a)}>0$. If $\la$ is large, then $v(x,0)\le u_0(x)$, whence $v(x,t)\le u(x,t)$.

In the next proposition we find a bound of the $L^1$-norm of $v(\cdot,t)$.
\begin{prop}\label{prop-bound of integral}
    Let $v$ be the solution to \eqref{eq-subsolution to bound integral}. There exist $C_0,\gamma>0$ such that
	\begin{equation*}\label{eq-bound integral}
	       \la^{\frac2{p-1}-N}	\int v(x,\la^2t)\,{\rm d}x \ge L-C_0L^p\big[\la^{-2\gamma}+ t^\gamma\big].
    \end{equation*}
\end{prop}
\begin{proof}
From the equation we see that
\[	
    \la^{\frac2{p-1}-N}	\int v(x,\la^2t)\,{\rm d}x=L-	\la^{\frac2{p-1}-N}\int_0^{\la^2t}\int v^p(x,s)\,{\rm d}x{\rm d}s.
\]
Therefore, it is enough to get a good bound from above for $\int_0^t\int v^p(x,s)\,{\rm d}x{\rm d}s$. To this end, we use that $v$ is a subsolution to the equation without absorption. Hence,
\begin{align*}
    v(x,t)&\le e^{-t}v(x,0)+\int W(x-y,t)v(y,0)\,{\rm d}y\\
    &\le L{\la^{N-\frac2{p-1}}}\Big[e^{-t}h(x)+\int W(x-y,t)h(y)\,{\rm d}y\Big],
\end{align*}
where $W$ is the regular part of the fundamental solution ($\Gamma(\cdot,t)=e^{-t}\delta +W(\cdot,t)$; see \cite{IR}). To estimate the last integral we will use the following bounds for $W$,
\begin{equation}\label{eq:bounds.W}
    W(x,t)\le Ct,\qquad W(x,t)\le C\frac t{|x|^{N+2}},\qquad W(x,t)\le C t^{-\frac N2}.
\end{equation}
The first two were proved in \cite{terra-wol} and the last one in~\cite{PR}.

If $|x|<2a$, using the first bound for $W$ in~\eqref{eq:bounds.W} we get
\[
    \int W(x-y,t)h(y)\,{\rm d}y\le Ct.
\]
On the other hand, if $|x|>2a$ and $|y|<a$, then $|x-y|>\frac{|x|}2$. Notice that $h$ vanishes if $|y|>a$. Hence, using the second bound in~\eqref{eq:bounds.W} we obtain
\begin{equation}\label{eq:integral.estimate}
    \int W(x-y,t)h(y)\,{\rm d}y\le C\frac t{|x|^{N+2}}.
\end{equation}
Using both integral estimates we conclude that~\eqref{eq:integral.estimate} holds for every $x\in\R^N$.	Finally, the third bound in~\eqref{eq:bounds.W} yields
\[
    \int W(x-y,t)h(y)\,{\rm d}y\le Ct^{-\frac N2}.
\]

Combining all the above estimates,
\[
    v(x,t)\le L{\la^{N-\frac2{p-1}}}\Big[e^{-t}h(x)+Ct^{-\frac N2}\chi_{|x|<\sqrt t}+C\frac t{|x|^{N+2}}\chi_{|x|>\sqrt t}\Big],
\]
whence
\begin{align*}
    \int v^p(x,t)\,{\rm d}x&\le CL^p{\la^{(N-\frac2{p-1})p}}\Big[e^{-pt}\int h^p(x)\,{\rm d}x+\int_{|x|<\sqrt t}t^{-p\frac N2}\,{\rm d}x+\int_{|x|>\sqrt t}\frac {t^p}{|x|^{p(N+2)}}\Big]\\
	&\le CL^p{\la^{(N-\frac2{p-1})p}}\big[e^{-pt}+t^{(1-p)\frac N2}\big].
\end{align*}
Thus, since $(p-1)\frac N2<1$, we have
\begin{align*}
    \la^{\frac2{p-1}-N}\int_0^{\la^2t}\int v^p(x,s)\,{\rm d}x{\rm d}s&\le CL^p\la^{(N-\frac2{p-1})(p-1)}\big[1+{(\la^2t)}^{1-\frac N2(p-1)}\big]\\
    &=CL^p\big[\la^{-2\gamma}+t^\gamma\big],
\end{align*}
where $\gamma=1-\frac N2(p-1)>0$.
\end{proof}
As a corollary we get the divergence of the mass of $U$ as $t\to0^+$.
\begin{coro}
    Let $U=\lim\limits_{j\to\infty}u_{\la_j}$ uniformly on compact sets. Then,
    \[
        \lim_{t\to0^+}\int U(x,t)\,{\rm d}x=\infty.
    \]
\end{coro}

\begin{proof}
Let $v$ be the solution to \eqref{eq-subsolution to bound integral}. By Proposition~\ref{prop-bound of integral},
\[
    \int u_{\la_j}(x,t)\,{\rm d}x\ge\int v_{\la_j}\,{\rm d}x\ge L-CL^p\big[\la_j^{-2\gamma}+t^\gamma\big].
\]
Therefore, since $u_{\la_j}(x,t)\le C|x|^{-\frac2{p-1}}$, then
\[
    \int_{|x|<R}\ul(x,t)\,{\rm d}x\ge \int \ul(x,t)\,{\rm d}x-C R^{N-\frac2{p-1}}\ge L-C\big[L^pt^\gamma+R^{N-\frac2{p-1}}\big].
\]
Since $p<1+\frac 2N$, letting first $j\to\infty$ and then $R\to\infty$ we obtain
\[
    \int_{\R^n}U(x,t)\,{\rm d}x \ge L-CL^pt^\gamma,
\]
whence
\[
    \liminf_{t\to0^+}\int_{\R^n}U(x,t)\,{\rm d}x \ge L.
\]	
As $L$ is arbitrary, the corollary is proved.
\end{proof}

This result allows us to identify $U$ as the unique  nonnegative very singular solution $V$ to the semilinear local heat equation~\eqref{eq:local.heat.absorption}.   Therefore, convergence is not restricted to a subsequence.
\begin{coro}
    The whole family $\{\ul\}$ converges uniformly on compact sets to $ V$.
\end{coro}

Now, we turn to the case $A>0$.  We know that all limit functions $U$ are solutions to the semilinear local heat equation with absorption~\eqref{eq:local.heat.absorption}. Hence, to identify the limit it is enough to show that they converge towards $A{|x|^{-2/(p-1)}}$ for all $x\neq0$ as $t\to0^+$. 
\begin{prop}\label{prop-U with A positive}	
    Let $u_0$ satisfy \eqref{eq-condition at infinity} with $A>0$. Let $U=\lim_{j\to\infty}u_{\la_j}$ uniformly in compact sets, with $\la_j\to\infty$. Then, for every $x\neq0$,
    \[
        U(x,t)\to A{|x|^{-2/(p-1)}}\quad\mbox{as }t\to0^+.
    \]
\end{prop}

\begin{proof}
The proof follows the lines of that of Proposition \ref{prop-U to 0}. As there, let  $x_0\neq0$, $M\ge 4^{3/(p-1)}|x_0|^{-2/(p-1)}\max\{A+\ep,C_0\}$, where $C_0$ is such that $u(x,t)\le C_0|x|^{-2/(p-1)}$ and $\ep>0$ is fixed but arbitrary. Then, if $\la$ is large there holds that $\ul\le M$ if $a-2d\la^{-1}<|x|$ and $a=\frac{|x_0|}4$. Now, let $\psi\in C^{2,\alpha}(\overline B)$ with $B=B_a(x_0)$ be such that $\psi=M$ close to $\partial B$, $\psi=(A+\ep)|x|^{-\frac2{p-1}}$ in $B_{a/2}(x_0)$ and let $v_\la$ be the solution to \eqref{eq-barrier in B}. Since $|x|^{\frac2{p-1}}u_0(x)\to A$ as $|x|\to\infty$,  then
\[
    (A-\ep){|x|^{-\frac2{p-1}}}\le \ul(x,0)\le (A+\ep){|x|^{-\frac2{p-1}}}\quad\mbox{in }B
\]
if $\la$ is large. Hence, by the choice of $M$, $\ul\le v_{\la}$ if $\la$ is large. As in Proposition~\ref{eq-barrier in B}, we know that $v_{\la}\to V$ uniformly in $\bar B\times[0,1]$ where $V$ is the solution to \eqref{eq-V}. Hence, $U\le V$ so that
\[
    \limsup_{t\to0^+}U(x_0,t)\le (A+\ep)|x_0|^{-\frac2{p-1}}.
\]
	
Now we prove that
\[
    \liminf_{t\to0^+}U(x_0,t)\ge (A-\ep)|x_0|^{-\frac2{p-1}}.
\]
To this end, we construct a subsolution. So, this time we let $w_\la$ be the solution to
\begin{equation*}\label{eq-subsolution w}
    \begin{cases}
		w_t-\L_\la w=-w^p&\mbox{in }B\times(0,1),\\
		w=0&\mbox{in }(\R^N\setminus B)\times(0,1),\\
		w(x,0)=\widehat\psi(x)&\mbox{in }B,
	\end{cases}
\end{equation*}
where $\widehat\psi\in C^{2,\alpha}(\overline B)$ with $B=B_a(x_0)$ is such that $\widehat\psi=0$ close to $\partial B$ and $\widehat\psi=(A-\ep)|x|^{-\frac2{p-1}}$ in $B_{a/2}(x_0)$.
	
In the next lemma, with arguments similar to those in~\cite{CER}, we prove that $w_\la\to W$ as $\la\to\infty$ where $W$ is the unique solution to
\begin{equation*}\label{eq-W}
	\begin{cases}
		W_t-\a\Delta W=-W^p&\mbox{in }B\times(0,1),\\
		W=0&\mbox{in }\partial B\times(0,1),\\
		W=\widehat\psi&\mbox{in }B.
	\end{cases}
\end{equation*}
As $\ul\ge w_\la$ in $B\times(0,1)$ if $\la$ is large, there holds that $U\ge W$ and hence,
\[
    \liminf_{t\to0^+} U(x_0,t)\ge (A-\ep)|x_0|^{-2/(p-1)}.
\]
As $\ep>0$ is arbitrary, the proposition is proved.	
\end{proof}

\begin{lema}
    Let $w_\la$ and $W$ as in the proof of Proposition~\ref{prop-U with A positive}. Then, $w_\la\to W$ as $\la\to\infty$ in $B_a(x_0)\times(0,1)$.
\end{lema}
\begin{proof}
Let $Z=w_\la-\bar W-\widetilde C\la^{-1}-\bar C\la^{-\alpha}t$ and $F=w_\la^p-\bar W^p$, where $\bar W$ is an extension of $W$ as a $C^{2,\alpha}(\R^N\times[0,1])$ function. Here, $\widetilde C$ is such that $|\bar W|\le\widetilde C\la^{-1}$ in $(B_{a+d\la^{-1}}(x_0)\setminus B_a(x_0))\times [0,1])$ and $\bar C$ is such that $| \a\Delta \bar W-\L_\la \bar W|\le \bar C\la^{-\alpha}$ in $\bar B\times[0,1]$. Then,
\[
    Z_t-\L_\la Z=(\L_\la\bar W-\a\Delta \bar W)-\bar C\la^{-\alpha}-w^p_\la+\bar W^p\le -F(x,t)\quad\mbox{in }B_a(x_0)\times(0,1).
\]
	
Moreover, $Z=-\bar W-\widetilde C\la^{-1}-\bar C\la^2t\le0$ in $(B_{a+d\la^{-1}}(x_0)\setminus B_a(x_0))\times(0,1)$, and $Z(x,0)=-\widetilde C\la^{-1}\le 0$ in $B_a(x_0)$. Since $Z>0$ implies that $w_\la>\bar W$, $Z$ is under the assumptions of Proposition~\ref{prop-comparison 2}. Hence, $Z\le0$, so that
\[
    \limsup_{\la\to\infty}(w_\la-W)\le0\quad\mbox{in }B_a(x_0)\times[0,1].
\]
	
Analogously,  $Z=\bar W-w_\la-\widetilde C\la^{-1}-\bar C\la^{-\alpha}t$, $F=\bar W^p-w_\la^p$ are under the assumptions of Proposition~\ref{prop-comparison 2}, so that
\[
    \limsup_{\la\to\infty}(W-w_\la)\le0\quad\mbox{in }B_a(x_0)\times[0,1]. \qedhere
\]
\end{proof}

This identifies the limit function $U$ as the unique nonnegative solution  $U_A$ to the semilinear local heat equation~\eqref{eq:local.heat.absorption}   such that $U(x,t)\to A|x|^{-2/(p-1)}$ as $t\to0^+$ if $x\neq0$. Hence, convergence is not restricted to a subsequence.
\begin{coro}
    The whole family $\{\ul\}$ converges uniformly on compact sets to $U_A$.
\end{coro}

\subsection{Asymptotic results}

Finally, using the selfsimilarity of the limit functions $V$ and~$U_A$, which makes them invariant under the scaling, we can translate the convergence of the family $\{u_\lambda\}$ in compact sets into large-time asymptotic results for the solution to~\eqref{eq-problem without sign}. We begin with the case $A=0$. 
\begin{teo}\label{teo-main}
    Let $p\in(1,1+\frac2N)$ and $u_0\ge0$, $u_0\in L^\infty(\R^N)$, such that $|x|^{\frac2{p-1}}u_0(x)\to0$ as $|x|\to\infty$. Let $V$ be the unique very singular solution  to~\eqref{eq:local.heat.absorption}   and $u$ the unique solution to \eqref{eq-problem without sign}. Then, for every $K>0$,
    \begin{equation}\label{eq:main.thm}
        t^{\frac1{p-1}}\|u(\cdot,t)-{V}(\cdot,t)\|_{L^\infty(B_{K\sqrt t})}\to0\quad\mbox{as }t\to\infty.
    \end{equation}
\end{teo}

\begin{proof}
Since the very singular solution $V$ to~\eqref{eq:local.heat.absorption} has the selfsimilar form $V(x,t)=t^{-\frac1{p-1}}F(|x|t^{-\frac12})$, see~\cite{BPT}, it is invariant under the scaling, $V_\la(x,t)=V(x,t)$. Hence, if $x=\lambda y$ and $t=\la^2$, then
\begin{align*}
    t^{\frac1{p-1}}\big(u(x,t)-V(x,t)\big)&=\lambda^{\frac2{p-1}}\big(u(\lambda y,\lambda^2)-V(\lambda y,\lambda^2)\big)= u_\la(y,1)-V_\la(y,1)\\
    &=u_\la(y,1)-V(y,1).
\end{align*}
The uniform convergence
\[
    \lim_{\lambda\to\infty}\|\ul(\cdot,1)-V(y,1)\|_{L^\infty(B_K)}
\]
yields then~\eqref{eq:main.thm}, since $|x|= |y|\lambda =|y|\sqrt{t}\le K\sqrt{t}$ if $x=\lambda y$, $t=\la^2$, and $|y|<K$.
\end{proof}
The proof for the case $A>0$ is similar.
\begin{teo}\label{teo-main2}
    Let $p\in(1,1+\frac2N)$ and $u_0\ge0$, $u_0\in L^\infty(\R^N)$, such that $|x|^{\frac2{p-1}}u_0(x)\to A>0$ as $|x|\to\infty$.  Let $U_A$ be the unique solution of the semilinear local heat equation~\eqref{eq:local.heat.absorption}   such that $U(x,t)\to A|x|^{-\frac2{p-1}}$ as $t\to0^+$ if $x\neq0$, and $u$ the unique solution to \eqref{eq-problem without sign}. Then, for every $R>0$,
    \[
        t^{\frac1{p-1}}\|u(x,t)-{U_A}(x,t)\|_{L^\infty(B_{R\sqrt t})}\to0\quad\mbox{as }t\to\infty.
    \]
\end{teo}

\begin{proof}
As in the proof of Theorem \ref{teo-main}, the uniqueness of the solution $U_A$ to the heat equation with absorption such that $U(x,t)\to A|x|^{-\frac2{p-1}}$ as $t\to0^+$ if $x\neq0$ implies that $U_A$ is invariant under the scaling, hence the result.
\end{proof}

\section*{Acknowledgments}

This research was supported by the European Union's Horizon 2020 research and innovation programme under the Marie Sklodowska-Curie grant agreement No.\,777822, and by the ICMAT-Centro de excelencia “Severo Ochoa” project,  CEX2023-001347-S, funded by MCIN/AEI/10.13039/501100011033 (Spain).

\noindent Fernando Quir\'os received also financial support from grants PID2020-116949GB-I00, PID2023-146931NB-I00, RED2022-134784-T and RED2024-153842-T, all of them funded by MCIN/AEI/10.13039/501100011033 (Spain), and from the Madrid Government (Comunidad de Madrid – Spain) under the multiannual Agreement with UAM in the line for the Excellence of the University Research Staff in the context of the V PRICIT (Regional Programme of Research and Technological Innovation).

\noindent Noemí Wolanski was also financially supported by ANPCyT PICT2016-1022 (Argentina).


\end{document}